\documentclass[11pt,reqno]{article}
\pdfoutput=1
\usepackage{geometry}
\usepackage{amsmath, amssymb, amsfonts,amsthm}
\usepackage{enumerate}
\usepackage{authblk}

\newtheorem{theorem}{Theorem}[section]
\newtheorem{proposition}[theorem]{Proposition}
\newtheorem{lemma}[theorem]{Lemma}
\theoremstyle{definition}
\newtheorem{corollary}[theorem]{Corollary}
\newtheorem{remark}[theorem]{Remark}

\usepackage{color}
\definecolor{darkblue}{rgb}{0.0,0.0,0.4}
\usepackage[pdfauthor={Tsogtgerel Gantumur},pdftitle={Regularity of solutions of the Navier-Stokes-alpha-beta equations with the wall-eddy boundary conditions},pdfsubject={Well-posedness and regularity},colorlinks,citecolor=darkblue,linkcolor=darkblue,urlcolor=darkblue]{hyperref}

\newcommand\eps{\varepsilon}

\newcommand\diam{\mathrm{diam}}
\newcommand\dom{\mathrm{dom}}
\newcommand\R{\mathbb{R}}
\newcommand\C{\mathbb{C}}

\title{Regularity of solutions of the Navier--Stokes-$\alpha\beta$ equations\\
with wall-eddy boundary conditions}

\author[1]{Nella Rotundo}
\author[2,3,4]{Gantumur Tsogtgerel}

\affil[1]{University of Florence}
\affil[2]{McGill University}
\affil[3]{National University of Mongolia}
\affil[4]{Institute of Mathematics and Digital Technology, Mongolian Academy of Sciences}

\date{\today}

\begin{document}

\maketitle

\begin{abstract}
We establish global well-posedness and regularity for the Navier--Stokes--$\alpha\beta$ system endowed with the wall--eddy boundary conditions proposed by Fried and Gurtin (2008). 
These conditions introduce a tangential vorticity traction proportional to wall vorticity and provide a continuum-mechanical model for near-wall turbulence. 
Our analysis begins with a variational formulation of the stationary fourth-order system, where we prove symmetry and a G{\aa}rding inequality for the associated bilinear form. 
We then verify Douglis--Nirenberg ellipticity and the Lopatinskii--Shapiro covering condition, establishing full Agmon--Douglis--Nirenberg regularity for the coupled system. 
Building on this framework, we derive a hierarchy of energy estimates for the nonlinear evolution equation, which yields global regularity, uniqueness, and stability. 
To our knowledge, this provides the first complete analytical treatment of the wall--eddy boundary model of Fried and Gurtin.
\end{abstract}

\tableofcontents

\section{Introduction}

Regularized models of incompressible flow add length scales and smoothing operators to the Navier--Stokes equations, stabilizing dynamics while preserving large-scale features. 
Classical $\alpha$-type models  such as Navier--Stokes--$\alpha$ and Leray--$\alpha$ are well understood under periodic or homogeneous Dirichlet boundary conditions; see, e.g., the abstract framework of Holst--Lunasin--Tsogtgerel~\cite{HLT10}. These settings, however, do not directly encode near-wall eddy physics.

Fried and Gurtin~\cite{FriedGurtin08,FriedGurtin11} proposed a two-length-scale theory in which a bulk filter scale $\alpha$ is coupled with a wall eddy scale $\beta$, accompanied by a tensorial boundary condition prescribing tangential vorticity traction proportional to the wall vorticity. Despite the compelling continuum mechanical derivation, a rigorous PDE analysis appears to be missing. The present paper fills this gap.

The model under consideration is called the {\em Navier-Stokes-$\alpha\beta$} (NS-$\alpha\beta$) equations, which are given in a bounded smooth domain $\Omega \subset \R^3$ by
\begin{equation}
\partial_t v - \Delta (1-\beta^2\Delta)u + (\nabla v) u + (\nabla u)^\top v + \nabla p = 0,
\end{equation}
where the filtered velocity $v$ is related to the fluid velocity $u$ by
\begin{equation}
v=(1-\alpha^2\Delta)u, \qquad \nabla\cdot u=0,
\end{equation}
with constants satisfying $\alpha>\beta>0$, see \cite[Eq.\thinspace(204)]{FriedGurtin08}. 
This system is coupled with the {\em wall-eddy boundary conditions}
\begin{equation}
u=0 \quad \text{and} \quad
\beta^2(1-n\otimes n)\left(\nabla\omega + \gamma (\nabla\omega)^\top\right)n = \ell \omega \quad \text{on } \partial\Omega,
\end{equation}
where $\omega = \nabla\times u$ is the vorticity, $n$ is the outward unit normal at $\partial\Omega$, and the parameters satisfy $|\gamma|\leq1$ and $\ell>0$, 
see \cite[Eq.\thinspace(203)]{FriedGurtin08}, corrected in \cite[Eq.\thinspace(C20)]{FriedGurtin11}.
We use the convention $(\nabla u)_{ij} = \partial_j u_i$, so that
$(\nabla u)v = (v\cdot\nabla)u$ for vector fields $u$ and $v$.

Our primary goal is to establish the global existence and regularity of solutions. We begin by studying the spatial principal part of the system, which is the linear stationary problem:
\begin{equation}\label{e:lin-stat}
\Delta^2u+\nabla p = f, \qquad \nabla\cdot u=0, \qquad \text{subject to the boundary conditions.}
\end{equation}
In \S\ref{s:lin} we introduce a pressure-free variational formulation of this linear problem, 
prove a G{\aa}rding-type inequality, and reconstruct the pressure by a de Rham type argument. 
We then place it in the Agmon--Douglis--Nirenberg framework to derive regularity statements, verifying both ellipticity and the covering condition.
This is the subject of \S\ref{s:reg}.
Finally, in \S\ref{s:evol}, we turn to the nonlinear evolution problem, and establish global well-posedness and regularity via a hierarchy of energy estimates.

\section{The linear stationary problem}
\label{s:lin}

\subsection{Preliminary properties of the bilinear form}

In order to analyse the linear stationary problem associated with the
NS–$\alpha\beta$ model we first introduce the bilinear form that encodes both
the interior dissipation and the wall–eddy boundary contribution.  
We define the relevant function
spaces and then derive several equivalent expressions for the form.  These
identities will be used repeatedly later.

Let $\Omega\subset\mathbb{R}^3$ be a bounded $C^2$ domain with outward unit normal $n$ at the boundary. 
The usual Sobolev spaces are denoted by $H^s(\Omega)$. 
With
\begin{equation}
\mathcal{V}=\{u\in \mathcal{D}(\Omega)^3:\nabla\cdot u=0\} ,
\end{equation}
we define $V$ to be the closure of $\mathcal{V}$ in $H^1(\Omega)^3$, and set $V^s=V\cap H^s(\Omega)^3$. 
Then we define the bilinear form $a$ on $H^2(\Omega)^3$ by
\begin{equation}
a(u,\phi) = \int_\Omega G:\nabla(\nabla\times\phi) + k \int_{\partial\Omega}(n\times\omega)\cdot\partial_n\phi.
\end{equation}
where 
$G = \nabla\omega + \gamma (\nabla\omega)^\top$,
$\omega = \nabla\times u$,
and $k=\ell/\beta^2$.

The next lemma rewrites $a(u,\phi)$ by repeated integration by parts.  The
first identity isolates the bulk operator $\Delta^2-\nabla\Delta(\nabla\cdot)$
acting on $u$, while the second form is more convenient for symmetric
estimates and for the G{\aa}rding inequality in \S\ref{ss:garding}.

\begin{lemma}
For $u,\phi\in C^\infty(\bar\Omega)^3$ with $\phi|_{\partial\Omega}=0$, we have
\begin{equation}\label{e:a-strong}
a(u,\phi) = \int_\Omega \big(\Delta^2u-\nabla\Delta(\nabla\cdot u)\big)\cdot\phi
   + \int_{\partial\Omega} (kn\times\omega-n\times Gn) \cdot \partial_n\phi ,
\end{equation}
and also
\begin{equation}\label{e:a-bilap}
\begin{split}
a(u,\phi) &= \int_\Omega \big(\Delta u\cdot\Delta\phi-\nabla(\nabla\cdot u)\cdot\nabla(\nabla\cdot\phi)\big)
   + \int_{\partial\Omega} (kn\times\omega-n\times Gn) \cdot \partial_n\phi \\
   &\quad-\int_{\partial\Omega} \Delta u\cdot\partial_n\phi  + \int_{\partial\Omega} \partial_n(\nabla\!\cdot u)(\nabla\!\cdot\phi) .
\end{split}
\end{equation}
\end{lemma}

\begin{proof}
For $u,\phi\in C^\infty(\bar\Omega)^3$ we have
\begin{equation}
\int_\Omega G : \nabla(\nabla\times\phi)
= - \int_\Omega (\nabla\cdot G)\cdot(\nabla\times\phi)
   + \int_{\partial\Omega} (G n)\cdot(\nabla\times\phi) ,
\end{equation}
and
\begin{equation}
- \int_\Omega (\nabla\cdot G)\cdot(\nabla\times\phi)
= - \int_\Omega \big(\nabla\times(\nabla\cdot G)\big)\cdot\phi
   + \int_{\partial\Omega} \big(n\times(\nabla\cdot G)\big)\cdot\phi .
\end{equation}
Here and in the following, we view a matrix field $G=(G_{ij})$ as a collection of row vectors, and use
the standard PDE convention
\begin{equation}
(\nabla\cdot G)_i = \sum_{j=1}^3 \partial_j G_{ij}.
\end{equation}
This convention ensures the product rule
$\nabla\cdot(Gv) = (\nabla\cdot G)\cdot v + G : \nabla v$
and is consistent with $\nabla\cdot(\nabla u)=\Delta u$.
Thus for $u,\phi\in C^\infty(\bar\Omega)^3$ with $\phi|_{\partial\Omega}=0$, we infer
\begin{equation}
\int_\Omega G : \nabla(\nabla\times\phi) = - \int_\Omega \big(\nabla\times(\nabla\cdot G)\big)\cdot\phi
   + \int_{\partial\Omega} (G n)\cdot(\nabla\times\phi) .
\end{equation}
For the volume term, we further compute
\begin{equation}
\nabla\cdot G = \Delta\omega + \gamma\nabla(\nabla\cdot\omega) = \Delta(\nabla\times u) = \nabla\times\Delta u ,
\end{equation}
where we have taken into account that $\nabla\cdot\omega=\nabla\cdot(\nabla\times u)=0$.
This yields
\begin{equation}
\nabla\times(\nabla\cdot G) 
= \nabla\times(\nabla\times\Delta u)
= \nabla(\nabla\cdot\Delta u)- \Delta^2u 
= \nabla\Delta(\nabla\cdot u)- \Delta^2u .
\end{equation}
Now we deal with the boundary term. 
Since $\phi=0$ on $\partial\Omega$, all tangential derivatives of $\phi$ vanish
along the boundary, and one checks directly that
\begin{equation}
\nabla\times\phi = n\times\partial_n\phi\qquad\text{on }\partial\Omega.
\end{equation}
Hence, using the identity $(a\times b)\cdot c
=(c\times a)\cdot b$ with $a=n$, $b=\partial_n\phi$, $c=Gn$, we obtain
\begin{equation}\label{eq:boundary-Gn-curlphi}
(Gn)\cdot(\nabla\times\phi)
=(Gn)\cdot(n\times\partial_n\phi)
=-(n\times Gn)\cdot\partial_n\phi
\qquad\text{on }\partial\Omega.
\end{equation}
In particular,
\begin{equation}\label{eq:boundary-Gn-curlphi-int}
\int_{\partial\Omega} (Gn)\cdot(\nabla\times\phi)
=-\int_{\partial\Omega} (n\times Gn)\cdot\partial_n\phi.
\end{equation}

To obtain \eqref{e:a-bilap}, we integrate by parts in the volume terms of
\eqref{e:a-strong}.  Writing $f=\Delta u$, we have
\begin{equation}
\int_\Omega \Delta^2 u\cdot\phi
 = \int_\Omega \Delta f_i\,\phi_i
 = -\int_\Omega \nabla f_i\cdot\nabla\phi_i
   + \int_{\partial\Omega} \partial_n f_i\,\phi_i,
\end{equation}
where the Einstein summation convention is assumed (repeated indices are summed over).
Since $\phi|_{\partial\Omega}=0$, the boundary term vanishes, and one more
integration by parts yields
\begin{equation}
\int_\Omega \Delta^2 u\cdot\phi
 = \int_\Omega f_i\,\Delta\phi_i
   - \int_{\partial\Omega} f_i\,\partial_n\phi_i
 = \int_\Omega \Delta u\cdot\Delta\phi
   - \int_{\partial\Omega} \Delta u\cdot\partial_n\phi.
\end{equation}
For the second volume term, let $h=\nabla\cdot u$ and $g=\Delta h$.  Then
\begin{equation}
-\int_\Omega \nabla\Delta(\nabla\cdot u)\cdot\phi
 = -\int_\Omega \nabla g\cdot\phi
 = \int_\Omega g\,(\nabla\cdot\phi)
   - \int_{\partial\Omega} g\,(\phi\cdot n),
\end{equation}
and again $\phi|_{\partial\Omega}=0$ removes the boundary term.  Using the
scalar Green identity once more,
\begin{equation}
\int_\Omega \Delta h\,(\nabla\cdot\phi)
 = -\int_\Omega \nabla h\cdot\nabla(\nabla\cdot\phi)
   + \int_{\partial\Omega} \partial_n h\,(\nabla\cdot\phi),
\end{equation}
so that
\begin{equation}
-\int_\Omega \nabla\Delta(\nabla\cdot u)\cdot\phi
 = -\int_\Omega \nabla(\nabla\cdot u)\cdot\nabla(\nabla\cdot\phi)
   + \int_{\partial\Omega} \partial_n(\nabla\cdot u)\,(\nabla\cdot\phi).
\end{equation}
Substituting these expressions into \eqref{e:a-strong} gives \eqref{e:a-bilap}.
\end{proof}

\subsection{Weak and strong solutions}
\label{ss:weak-strong}

In this subsection we relate the variational structure induced by $a(\cdot,\cdot)$
to the stationary fourth–order system \eqref{e:lin-stat}. We first specify a pressure–free
weak formulation and then show how to reconstruct the pressure and the strong
boundary conditions from such a solution.

Our proposed weak formulation of \eqref{e:lin-stat} is as follows:
Find $u \in V^2$ such that
\begin{equation}\label{e:weak}
a(u, \phi) = \langle f, \phi\rangle \qquad\text{for all}\quad \phi \in V^2 ,
\end{equation}
where $f\in (V^2)'$ is given.


\begin{theorem}\label{t:weak-strong}
(a) Suppose that $u\in V^2$ satisfies \eqref{e:weak}, and let $f\in H^{-2}(\Omega)^3$. Then there exists $p\in H^{-1}(\Omega)$ such that
\begin{equation}
\Delta^2u+\nabla p = f .
\end{equation}
(b) In addition, assume that $u\in V^4$ and $f\in L^{2}(\Omega)^3$.
Then $p\in H^1(\Omega)$ and 
\begin{equation}
u=0 \quad \textrm{and} \quad (1-n\otimes n)Gn=k\omega \quad \textrm{on}\;\partial\Omega.
\end{equation}
\end{theorem}

\begin{proof}
(a)
We can define a linear functional $L$ on test fields, as follows:
$$
\langle L, \psi \rangle = \langle f, \psi\rangle - a(u, \psi), \qquad \psi \in \mathcal{D}(\Omega)^3.
$$
From the definition of our weak solution, $L$ vanishes on $\mathcal{V}$.
Since $L$ vanishes on all divergence-free test fields, by the classical
Helmholtz decomposition
there exists a scalar distribution $p\in H^{-1}(\Omega)$ such that $L=\nabla p$,  cf. \cite[\S1.4]{Temam}.
Now if $u\in C^\infty(\bar{\Omega})^3$, then \eqref{e:a-bilap} gives
\begin{equation}
a(u,\phi) = \int_\Omega \big(\Delta u\cdot\Delta\phi-\nabla(\nabla\cdot u)\cdot\nabla(\nabla\cdot\phi)\big) ,
\qquad \phi\in\mathcal{D}(\Omega)^3,
\end{equation}
and as $u\mapsto a(u,\phi)$ is continuous as a map $H^2(\Omega)^3\to\R$ for any fixed test field $\phi\in\mathcal{D}(\Omega)^3$, 
this formula is valid for $u\in H^2(\Omega)^3$ by density.
In particular, for $u\in V^2$, we have
\begin{equation}
a(u,\phi) = \int_\Omega \Delta u\cdot\Delta\phi = \langle\Delta^2u,\phi\rangle .
\end{equation}
Thus we find that
$$
\langle f - \Delta^2 u, \psi \rangle = \langle \nabla p, \psi \rangle , \qquad\text{for any}\ \psi \in \mathcal{D}(\Omega)^3 ,
$$
meaning that the equation $\Delta^2 u + \nabla p = f$ holds in the sense of distributions.

(b)
If $u\in V^4$ and $f\in L^{2}(\Omega)^3$, then $\nabla p=f-\Delta^2 u\in L^2(\Omega)^3$ and hence $p\in H^1(\Omega)^3$.
Moreover, $\Delta^2 u + \nabla p = f$ holds now almost everywhere pointwise in $\Omega$.
The space $V^4$ already incorporates the homogeneous Dirichlet boundary condition, meaning that $u=0$ on $\partial\Omega$ almost everywhere.
Arguing by density as above, we see that \eqref{e:a-strong} is valid for $u\in H^4(\Omega)^3$.
Then taking into account that $\nabla\cdot u=0$, we infer from the weak formulation \eqref{e:weak} that
\begin{equation}
\int_\Omega (\Delta^2u)\cdot\phi
   + \int_{\partial\Omega} (kn\times\omega-n\times Gn) \cdot \partial_n\phi = \int_\Omega f\cdot\phi ,
\end{equation}
for all $\phi\in C^\infty(\bar\Omega)^3$ with $\nabla\cdot\phi=0$ and $\phi|_{\partial\Omega}=0$.
Since 
\begin{equation}
\int_\Omega (f-\Delta^2u)\cdot\phi = \int_\Omega (\nabla p)\cdot\phi = \int_{\partial\Omega}pn\cdot\phi - \int_\Omega p\nabla\cdot\phi = 0 ,
\end{equation}
for such test fields, we have
\begin{equation}\label{e:weak-strong-pf-1}
\int_{\partial\Omega} (kn\times\omega-n\times Gn) \cdot \partial_n\phi = 0 ,
\end{equation}
for all $\phi\in C^\infty(\bar\Omega)^3$ with $\nabla\cdot\phi=0$ and $\phi|_{\partial\Omega}=0$.
Note that the boundary vector field $B = kn\times\omega-n\times Gn$ is tangential on $\partial\Omega$.  
It is a standard fact (see, e.g., the construction of
solenoidal extensions in \cite[\S1.4]{Temam}) that the set
\begin{equation}
\bigl\{\partial_n\phi|_{\partial\Omega} :
    \phi\in C^\infty(\bar\Omega)^3,\ \phi|_{\partial\Omega}=0,\ 
    \nabla\cdot\phi=0\bigr\}
\end{equation}
is dense in the space of smooth tangential vector fields on $\partial\Omega$.
Hence \eqref{e:weak-strong-pf-1} implies
\begin{equation}
\int_{\partial\Omega} B\cdot\tau = 0
\end{equation}
for all smooth tangential $\tau$, hence $B=0$ on $\partial\Omega$, i.e.
\begin{equation}
n\times(k\omega - Gn) = 0 \quad\text{on }\partial\Omega.
\end{equation}
This says that the tangential component of $Gn-k\omega$ vanishes
\begin{equation}
(1-n\otimes n)(k\omega - Gn) = 0 \quad\text{on }\partial\Omega.
\end{equation}
However, local computations show that $\omega$ is tangential because $u=0$ on $\partial\Omega$,
and hence $(1-n\otimes n)\omega=\omega$,
completing the proof.
\end{proof}

\begin{remark}
In part~\textup{(a)} of Theorem~\ref{t:weak-strong} we work entirely at
the distributional level; the argument only uses integration by parts in the
volume and basic trace properties of $H^2(\Omega)^3$, so a bounded Lipschitz
domain would already be sufficient.  Part~\textup{(b)} uses the fact that
$u\in V^4$ has well-defined traces of $\nabla\times u$ and $G(u)n$ on
$\partial\Omega$ and that these traces are compatible with the strong wall--eddy
boundary condition.  Again, $\partial\Omega\in C^2$ is ample for these trace
statements and for the elliptic regularity invoked here.  In later sections,
when we appeal to Agmon--Douglis--Nirenberg regularity for the fourth-order
system, higher boundary regularity (e.g.\ $\partial\Omega\in C^{4}$) will be
assumed explicitly in the corresponding theorems.
\end{remark}

\subsection{Symmetry and continuity}

Having identified the appropriate weak formulation, we now collect some basic
functional analytic properties of the bilinear form $a(\cdot,\cdot)$.  In
particular, we establish its continuity on $H^2(\Omega)^3$ and its symmetry on
vector fields with homogeneous Dirichlet data.

\begin{theorem}
\label{thm:sym}
(a)
There exists $C>0$ such that
\begin{equation}
      |a(u,\phi)| \le C\,\|u\|_{H^2(\Omega)}\,\|\phi\|_{H^2(\Omega)}
      \qquad\forall\,u,\phi\in H^2(\Omega)^3.
\end{equation}
(b)
For all $u,\phi\in H^2(\Omega)^3\cap H^1_0(\Omega)^3$ one has
\begin{equation}
      a(u,\phi) = a(\phi,u).
\end{equation}
\end{theorem}

\begin{proof}
(a)
For $u\in H^2(\Omega)^3$ we have $\omega=\nabla\times u\in H^1(\Omega)^3$ and $\nabla\omega\in L^2(\Omega)^{3\times 3}$. 
Thus $G\in L^2(\Omega)^{3\times 3}$ with
\begin{equation}
\|G\|_{L^2(\Omega)}
 \lesssim \|\nabla\omega(u)\|_{L^2(\Omega)}
 \lesssim \|u\|_{H^2(\Omega)}.
\end{equation}
Here and in what follows, we write $A\lesssim B$ for $A\le C\,B$ with $C>0$ that may have different values at its different occurrences.
For $\phi\in H^2(\Omega)^3$ one has
$\nabla(\nabla\times\phi)\in L^2(\Omega)^{3\times 3}$ and
\begin{equation}
\|\nabla(\nabla\times\phi)\|_{L^2(\Omega)}
 \lesssim \|\phi\|_{H^2(\Omega)}.
\end{equation}
Hence
\begin{equation}
\biggl|\int_\Omega G:\nabla(\nabla\times\phi)\,dx\biggr|
 \le \|G\|_{L^2}\,\|\nabla(\nabla\times\phi)\|_{L^2}
 \lesssim \|u\|_{H^2}\,\|\phi\|_{H^2}.
\end{equation}

Now we turn to the boundary term. 
Since
$\omega\in H^1(\Omega)^3$, we have $\omega|_{\partial\Omega}\in H^{1/2}(\partial\Omega)^3$, and similarly
$\partial_n\phi\in H^{1/2}(\partial\Omega)^3$. Thus
\begin{equation}
\biggl|\int_{\partial\Omega} (n\times\omega)\cdot\partial_n\phi\,dS\biggr|
 \lesssim \|\omega\|_{H^1(\Omega)}\,\|\phi\|_{H^2(\Omega)}
 \lesssim \|u\|_{H^2(\Omega)}\,\|\phi\|_{H^2(\Omega)}.
\end{equation}
Combining the two estimates yields (a).

(b)
The symmetry of the volume integral follows from
\begin{equation}
(\nabla\omega+\gamma(\nabla\omega)^\top):\nabla\psi 
=
\nabla\omega:\nabla\psi +\gamma(\nabla\omega)^\top:\nabla\psi 
=
\nabla\omega:(\nabla\psi +\gamma(\nabla\psi)^\top) ,
\end{equation}
where $\psi=\nabla\times\phi$, and where have used the fact that $A^\top : B = A : B^\top$ for any pair of square matrices $A,B\in\mathbb{R}^{3\times 3}$.
The boundary term is also symmetric, since $u=0$ on $\partial\Omega$ implies $\omega=n\times\partial_nu$, and so
\begin{equation}
(n\times\omega)\cdot\partial_n\phi
= (n\times n\times\partial_nu)\cdot\partial_n\phi
= -(n\times\partial_nu)\cdot(n\times\partial_n\phi) ,
\end{equation}
where we have used the triple product identity 
$(a\times b)\cdot c = b\cdot(c\times a)$ with $a=n$, $b=n\times\partial_n u$ and $c=\partial_n\phi$.
Thus $a(u,\phi)=a(\phi,u)$ for all $u,\phi\in V^2$.
\end{proof}

\begin{remark}
In the preceding theorem we use only:
\begin{itemize}
  \item standard trace theorems for $H^2(\Omega)$ and $H^1(\Omega)$, and
  \item elementary integration by parts identities in the volume and on
        $\partial\Omega$.
\end{itemize}
In particular, the proof does not rely on any delicate elliptic regularity,
and the assumption $\partial\Omega\in C^2$ can be relaxed here: a bounded
Lipschitz domain already suffices to justify the traces and Green formulas used in the argument.
We keep the global standing assumption $\partial\Omega\in C^2$ for notational
uniformity throughout the section.
\end{remark}

\subsection{Coercivity}
\label{ss:garding}

To use the form $a(\cdot,\cdot)$ as the starting point for the operator
theory in \S\ref{ss:operator} and for the energy estimates in \S\ref{s:evol}, we require a
Gårding type lower bound.  The following theorem provides such an estimate
and shows, in particular, that $a(\cdot,\cdot)$ controls the full $H^2$–norm
of divergence-free fields.

\begin{theorem}\label{t:garding}
      There exist constants $c_0>0$ and $\gamma_0\ge 0$, depending only
      on $\Omega$, $\gamma$ and $k$, such that
      \begin{equation}
      \label{eq:garding}
      a(u,u) + \gamma_0 \|u\|_{L^2(\Omega)}^2 \ge c_0 \|u\|_{H^2(\Omega)}^2
      \qquad\text{for all}\ u\in V^2 .
      \end{equation}
\end{theorem}

\begin{proof}
Let us treat the volume term first. We consider 3 cases, depending on the value of $\gamma$.

{\em Case $|\gamma|<1$}.
Using the Cauchy--Schwarz inequality
\begin{equation}
\bigl|(\nabla\omega)^\top:\nabla\omega\bigr|  \le |\nabla\omega|^2,
\end{equation}
we obtain pointwise
\begin{equation}
G:\nabla\omega
 \ge (1-|\gamma|)\,|\nabla\omega|^2.
\end{equation}
Integrating over $\Omega$ gives
\begin{equation}
\label{eq:garding-omega}
\int_\Omega G:\nabla\omega
 \ge (1-|\gamma|)\|\nabla\omega\|_{L^2(\Omega)}^2 \gtrsim \|\nabla\times\omega\|_{L^2(\Omega)}^2 = \|\Delta u\|_{L^2(\Omega)}^2 \gtrsim \|u\|_{H^2(\Omega)}^2 ,
\end{equation}
since $u$ satisfies the homogeneous Dirichlet boundary condition.

{\em Case $\gamma=-1$}.
For any square matrix $A$, we have
\begin{equation}
(A-A^\top):A 
= (A-A^\top):\big(\frac{A+A^\top}2 + \frac{A-A^\top}2\big)
= \frac12(A-A^\top):(A-A^\top) ,
\end{equation}
because $(A-A^\top):(A+A^\top)=0$.
If $A=\nabla\omega$ then the elements of $A-A^\top$ are exactly the elements of $\nabla\times\omega$ counted twice,
and hence
\begin{equation}
G:\nabla\omega = 
\frac12(\nabla\omega-\nabla\omega^\top):(\nabla\omega-\nabla\omega^\top)
= |\nabla\times\omega|^2 .
\end{equation}
Then taking into account
\begin{equation}
\nabla\times\omega = \nabla\times\nabla\times u = \nabla(\nabla\cdot u) - \Delta u = - \Delta u ,
\end{equation}
we have
\begin{equation}
\int_\Omega G:\nabla\omega = \|\nabla\times\omega\|_{L^2(\Omega)}^2 = \|\Delta u\|_{L^2(\Omega)}^2 \gtrsim \|u\|_{H^2(\Omega)}^2 ,
\end{equation}
since $u$ satisfies the homogeneous Dirichlet boundary condition.

{\em Case $\gamma=1$}.
Similarly to the preceding case, we compute
\begin{equation}
G:\nabla\omega = 
\frac12(\nabla\omega+\nabla\omega^\top):(\nabla\omega+\nabla\omega^\top)
= 2|\eps(\omega)|^2 ,
\end{equation}
where 
\begin{equation}
\eps(\omega) = \frac12(\nabla\omega+\nabla\omega^\top) ,
\end{equation}
is the symmetric gradient of $\omega$.
By the second Korn inequality (see, e.g., \cite[Thm.~2.5-2]{Ciarlet} or
\cite[Thm.~10.1.1]{BrennerScott}), we have
\begin{equation}
\|\omega\|_{H^1(\Omega)} \lesssim \|\eps(\omega)\|_{L^2(\Omega)} + \|\omega\|_{L^2(\Omega)} ,
\end{equation}
and so
\begin{equation}
\int_\Omega G:\nabla\omega = 2\|\eps(\omega)\|_{L^2(\Omega)}^2 \ge c \|\omega\|_{H^1(\Omega)}^2 - C \|\omega\|_{L^2(\Omega)}^2 ,
\end{equation}
with $c>0$.
The first term on the right hand side can be bounded as
\begin{equation}
\|\omega\|_{H^1(\Omega)}\geq\|\nabla\times\omega\|_{L^2(\Omega)} = \|\Delta u\|_{L^2(\Omega)} \gtrsim \|u\|_{H^2(\Omega)} ,
\end{equation}
and the second term as
\begin{equation}
\|\omega\|_{L^2(\Omega)} \leq \|u\|_{H^1(\Omega)} \leq \eps \|u\|_{H^2(\Omega)} + C_\eps \|u\|_{L^2(\Omega)} .
\end{equation}

Now we will deal with the boundary term. 
By the trace theorem and an interpolation inequality, for any $\varepsilon>0$ there exists $C_\varepsilon>0$ such that
\begin{equation}
\int_{\partial\Omega} |n\times\partial_nu|^2
 \le C \|u\|_{H^{3/2}(\Omega)}^2
 \le \varepsilon \|u\|_{H^2(\Omega)}^2
      + C_\varepsilon\,\|u\|_{L^2(\Omega)}^2.
\end{equation}
This completes the proof.
\end{proof}

The preceding estimate is sufficient for our purposes, but it is useful to
record two variants that clarify the role of the parameters and the boundary
regularity.

\begin{remark}
The bilinear form $a$ is strictly coercive (i.e., $C=0$) if the parameters satisfy a smallness condition, for instance, of the form
$$
\frac{\ell}{\beta}<\frac{\delta\beta}{\diam(\Omega)} \quad \text{for some } \delta > 0.
$$
\end{remark}

\begin{remark}
The G\r{a}rding inequality as stated, with control of the
full $H^2$-norm of $u$, relies on global elliptic regularity for the Poisson
problem with Dirichlet boundary conditions.  We do not attempt to optimize the
boundary regularity here; the standing assumption $\partial\Omega\in C^2$ is
more than sufficient for the estimates used in this section.
\end{remark}

\subsection{The elliptic operator}
\label{ss:operator}

The $L^2$–closure of divergence-free test fields is denoted by $H$.  The
$L^2$–inner product and norm on $H$ are written $(\cdot,\cdot)$ and
$\|\cdot\|$, and $P:L^2(\Omega)^3\to H$ is the Leray projector.  Recall
that $V$ is the closure of $\mathcal{V}$ in $H^1(\Omega)^3$ and that
$V^2=V\cap H^2(\Omega)^3$.


The bilinear form $a:V^2\times V^2\to\mathbb{R}$ is continuous and symmetric
and satisfies the G\r{a}rding inequality
\begin{equation}\label{eq:garding-shifted}
  a(u,u) + \gamma_0\|u\|^2 \,\ge c_0\|u\|_{H^2(\Omega)}^2,
  \qquad u\in V^2,
\end{equation}
for some constants $c_0>0$ and $\gamma_0\ge0$, cf.\ Theorem~\ref{t:garding}.
In particular $a(\cdot,\cdot)$ is a densely defined, closed, lower bounded
symmetric form on $H$ in the sense of Kato
\cite{Kato} and Lions--Magenes \cite{LionsMagenes}.

\begin{proposition}\label{p:A-sa-gen}
There exists a unique self-adjoint operator $A:\dom(A)\subset H\to H$ such
that
\begin{equation}\label{eq:A-form}
  (Au,\phi) \,=\, a(u,\phi)
  \qquad\text{for all }u\in\dom(A),\ \phi\in V^2,
\end{equation}
where
\begin{equation}
  \dom(A)
  := \bigl\{u\in V^2 : \exists f\in H\text{ with }a(u,\phi)=(f,\phi)
       \ \forall\,\phi\in V^2\bigr\} .
\end{equation}
The spectrum of $A$ is bounded from below,
\begin{equation}
  (Au,u)\ge -\gamma_0\|u\|^2,\qquad u\in\dom(A),
\end{equation}
and the resolvent $(A+\eta I)^{-1}:H\to H$ is compact for every $\eta\ge\gamma_0$.
\end{proposition}

\begin{proof}
By \eqref{eq:garding-shifted} the shifted form
\begin{equation}
  a_{\gamma_0}(u,\phi):=a(u,\phi)+\gamma_0(u,\phi)
\end{equation}
is symmetric, continuous on $V^2\times V^2$, and strictly coercive.  The general form
theory (see \cite{Kato} or \cite{LionsMagenes}) yields a unique self-adjoint operator
$\widetilde A$ associated with $a_{\gamma_0}$ by
$(\widetilde Au,\phi)=a_{\gamma_0}(u,\phi)$ for
$u\in\dom(\widetilde A)$, $\phi\in V^2$.  Setting $A:=\widetilde A-\gamma_0I$
gives \eqref{eq:A-form}, and $(Au,u)\ge-\gamma_0\|u\|_H^2$ follows from the
coercivity of $a_{\gamma_0}$.  The compactness of the embedding
$V^2\hookrightarrow H$ implies that $(\widetilde A+\eta I)^{-1}$ is compact
for every $\eta\ge0$, hence $(A+\eta I)^{-1}$ is compact for $\eta\ge\gamma_0$.
\end{proof}

By the spectral theory for compact self-adjoint operators
(e.g., \cite{ReedSimon1}), we obtain:

\begin{corollary}\label{c:spectrum-A}
The spectrum of $A$ consists of a nondecreasing sequence of real
eigenvalues
\begin{equation}
  \lambda_1\le\lambda_2\le\cdots\le\lambda_n\le\cdots,\qquad
  \lambda_n\to+\infty\ \text{as }n\to\infty,
\end{equation}
each of finite multiplicity, and $\lambda_1>-\gamma_0$.  The corresponding
eigenfunctions $\{v_n\}_{n\ge1}\subset\dom(A)$ form an orthonormal basis of
$H$ and an orthogonal basis of $V^2$ with respect to $a(\cdot,\cdot)$:
\begin{equation}
  (v_n,v_m)=\delta_{nm}, \qquad a(v_n,v_m)=\lambda_n\delta_{nm}.
\end{equation}
In particular, every $u\in H$ admits the expansion
\begin{equation}
  u = \sum_{n=1}^\infty (u,v_n)v_n
\end{equation}
with convergence in $H$.
\end{corollary}

\begin{corollary}\label{c:ell-ex}
With the right hand side $f\in H$, the weak formulation \eqref{e:weak} of the linear stationary problem \eqref{e:lin-stat} can be stated as
\begin{equation}\label{e:op}
Au = f .
\end{equation}
If $0$ happens to be an eigenvalue of $A$, then $f$ must be orthogonal to the (finite dimensional) kernel of $A$ in order for this equation to have a solution,
and the solution will be determined up to this finite dimensional kernel.
\end{corollary}

\begin{remark}
For any bounded Borel function $g:\mathbb{R}\to\mathbb{C}$ we can now define
the bounded operator
\begin{equation}\label{eq:g-of-A-gen}
  g(A)u := \sum_{n=1}^\infty g(\lambda_n)(u,v_n)v_n,
  \qquad u\in H,
\end{equation}
whenever $\sup_n|g(\lambda_n)|<\infty$.  In particular, for
$\sigma\in\mathbb{R}$ we obtain fractional powers $A^\sigma$ and the
associated scale of Hilbert spaces
\begin{equation}
  H_\sigma := \dom(A^{\sigma/2}),\qquad
  \|u\|_{H_\sigma} := \bigl\|A^{\sigma/2}u\bigr\| .
\end{equation}
These abstract spaces will be identified in the next section with
concrete Sobolev spaces via elliptic regularity; for instance, we will show that
$\dom(A)\subset V^4$ and that $\|A u\|$ is equivalent to the $H^4$–norm of
$u$ on $\dom(A)$.
\end{remark}

\section{Elliptic regularity}
\label{s:reg}

\subsection{Discussion of the Agmon-Douglis-Nirenberg theory}
\label{ss:adn}

The existence of a weak solution $u\in V^2$ obtained from the bilinear form $a(\cdot,\cdot)$
provides a natural starting point for our analysis of regularity. In order to control
higher derivatives of $u$ (and, ultimately, of the nonlinear evolution), we need estimates
that are uniform in the order of differentiation. The appropriate tool is the
Agmon--Douglis--Nirenberg (ADN) theory for general elliptic systems with boundary
conditions, cf.~\cite{ADN64,WRL}. In this section we verify that the stationary problem for
$(u,p)$ fits into the ADN framework and then invoke the abstract regularity result.
Throughout this section we assume that $\partial\Omega$ is of class $C^{4+m}$, where
$m\ge0$ will be specified in Theorem~\ref{t:ell-reg}.

In light of Theorem~\ref{t:weak-strong},
the stationary system can be written as a system of four equations for the four unknowns $(u_1, u_2, u_3, p)$:
\begin{equation}
    \Delta^2 u_i + \partial_i p = f_i, \quad i=1,2,3, \qquad 
    \nabla \cdot u = 0 .
\end{equation}
This can be expressed in matrix operator form $L(\partial) U = F$, 
where
\begin{equation}
L(\partial) = 
\begin{bmatrix}
\Delta^2 & 0 & 0 & \partial_1 \\
0 & \Delta^2 & 0 & \partial_2 \\
0 & 0 & \Delta^2 & \partial_3 \\
\partial_1 & \partial_2 & \partial_3 & 0
\end{bmatrix} ,
\qquad
U = 
\begin{bmatrix}
u_1 \\ u_2 \\ u_3 \\ p
\end{bmatrix} ,
\qquad
F = 
\begin{bmatrix}
f_1 \\ f_2 \\ f_3 \\ 0
\end{bmatrix} .
\end{equation}
To apply the ADN theory, we must first assign orders to the rows and columns of this system. 
Let $s_i$ be the order for the $i$-th equation and $t_j$ be the order for the $j$-th variable. 
We choose $s = (s_1, s_2, s_3, s_4) = (4, 4, 4, 1)$ and $t = (t_1, t_2, t_3, t_4) = (0, 0, 0, -3)$. 
With this choice, the order of the differential operator $L_{ij}$ is at most $s_i + t_j$.
By convention, the order of $0$ is $-\infty$.

Next, the principal symbol $\hat L^0(\xi)$ is formed by taking, in each entry $L_{ij}(\partial)$, the terms of order \emph{exactly} $s_i+t_j$, and replacing $\partial_k$ with $i\xi_k$.
Then the {\em Douglis-Nirenberg ellipticity} requires that the matrix $\hat L^0(i\xi)$ is invertible for all $\xi\neq0$.
In our case, it is easy to see that $\hat L^0(\xi)=L(i\xi)$, and that this ellipticity condition is satisfied.
For completeness, we will spell out the confirmation of ellipticity in the next subsection (\S\ref{ss:ver-ell}).

Finally, we must verify that the boundary conditions satisfy the {\em Lopatinskii-Shapiro (or covering) condition}. We have five boundary conditions in total:
\begin{itemize}
    \item $u_j = 0$ on $\partial\Omega$, for $j=1,2,3$. (Three conditions of order 0).
    \item $(1-n\otimes n)Gn - k\omega = 0$ on $\partial\Omega$. This is a vector equation on the tangent plane, providing two independent scalar conditions involving the second derivatives of $u$.
\end{itemize}
Let us write these conditions as $B(x,\partial)U=0$, where $B(x,\partial)$ is a $5\times4$ matrix consisting of differential operators, with coefficients possibly depending on $x\in\partial\Omega$.
The orders of the boundary conditions (i.e., rows of $B$) are $r=(0,0,0,2,2)$.
The principal part $B^0$ is formed by taking, in each entry $B_{ij}(x,\partial)$, the terms of order \emph{exactly} $r_i+t_j$.

We flatten the boundary near a point and work in the half-space $\R^3_+ = \{(y_1,y_2,t) : t > 0\}$ with the boundary at $t=0$.
We apply the Fourier transform in the tangential directions $y = (y_1, y_2) \to \eta = (\eta_1, \eta_2)$. 
The PDE system $L(\partial)U=0$ becomes a system of ordinary differential equations in the normal variable $t$ for the transformed function $\hat{U}(\eta,t)$:
\begin{equation}\label{e:covering-pde}
L(i\eta, \frac{d}{dt})\hat{U} = 0.
\end{equation}
Then, tailored to our case, the \emph{covering condition} is:
For any fixed $\eta\in\R^2\setminus\{0\}$, the only solution $\hat U(\eta,t)$ of the ODE system \eqref{e:covering-pde} on the half–line $t>0$ which
\begin{itemize}
  \item decays as $t\to+\infty$, and
  \item satisfies all homogeneous principal boundary conditions
	\begin{equation}\label{e:covering-bc}
    	B^0(i\eta, \frac{d}{dt})\hat U|_{t=0}=0,
	\end{equation}
\end{itemize}
is the trivial solution $\hat U(\eta,\cdot)\equiv0$.
We will verify the covering condition for our system in \S\ref{ss:ver-cov}.

Once ellipticity and the Lopatinskii--Shapiro (covering) condition have been
verified for the pair $(L,B)$, the ADN theory yields full interior and boundary
regularity for weak solutions. We state here the consequence that will be used later
for the nonlinear problem.


\begin{theorem}\label{t:ell-reg}
Assume that $\partial\Omega$ is of class $C^{4+m}$ for an integer $m\ge0$. Let
$f\in H^m(\Omega)^3$ and let $(u,p)$ be a weak solution of the stationary problem
\begin{equation}\label{eq:strong-stationary}
\begin{cases}
\Delta^2 u + \nabla p = f,& \text{in }\Omega,\\
\nabla\cdot u = 0,& \text{in }\Omega,\\
u = 0,& \text{on }\partial\Omega,\\
(1-n\otimes n)Gn = k\omega,& \text{on }\partial\Omega,
\end{cases}
\end{equation}
in the sense of Theorem~\ref{t:weak-strong} (i.e.\ $u\in V^2$, $f\in H^{-2}(\Omega)^3$,
and there exists $p\in H^{-1}(\Omega)$ such that $\Delta^2u+\nabla p=f$ in $\mathcal D'$
and \eqref{eq:strong-stationary}$_4$ holds in the weak sense).
Then
\begin{equation}
u\in H^{m+4}(\Omega)^3,\qquad p\in H^{m+3}(\Omega),
\end{equation}
and there exists a constant $C=C(m,\Omega,\alpha,\beta,\gamma,\ell)>0$ such that
\begin{equation}\label{eq:ADN-estimate}
\|u\|_{H^{m+4}(\Omega)} + \|p\|_{H^{m+3}(\Omega)}
\le C\bigl(\|f\|_{H^{m}(\Omega)} + \|u\|_{L^2(\Omega)}\bigr).
\end{equation}
In particular, for $f\in L^2(\Omega)^3$ we obtain $u\in V^4$ and $p\in H^3(\Omega)$.
\end{theorem}

\begin{proof}
By Proposition~\ref{p:dn-ell} and Theorem~\ref{t:ls-cov} below, the pair $(L,B)$ is Douglis--Nirenberg elliptic and the boundary conditions satisfy
the Lopatinskii--Shapiro covering condition. The regularity result
\eqref{eq:ADN-estimate} then follows from the general ADN theorem
(e.g., \cite[Thm.~10.5.1]{WRL}). The passage from the
pressure-free variational formulation to the coupled system with pressure
was already established in Theorem~\ref{t:weak-strong}, so the ADN result applies
to the pair $(u,p)$ constructed there.
\end{proof}

This theorem immediately yields the following inclusion for the domain of $A$.

\begin{corollary}\label{c:dom-A}
If $\partial\Omega$ is of class $C^{4}$, then the operator $A$
introduced in \S\ref{ss:operator} satisfies
\begin{equation}
\dom(A) \subset V^{4},
\end{equation}
and there exists a constant $C>0$ such that
\begin{equation}\label{eq:A-graph-H4}
  \|u\|_{H^{4}(\Omega)}
  \le C\bigl(\|Au\|+\|u\|\bigr)
  \qquad\text{for all }u\in\dom(A).
\end{equation}
\end{corollary}

\begin{proof}
Let $u\in\dom(A)$. By definition of $A$ there exists $f\in H$ such that
\begin{equation}
a(u,\phi)=(f,\phi)_{H}
\qquad\text{for all }\phi\in V^{2}.
\end{equation}
By Theorem~\ref{t:weak-strong} there is a pressure $p\in H^{-1}(\Omega)$ such that the pair
$(u,p)$ is a weak solution of \eqref{eq:strong-stationary} with right-hand side $f$.  Since
$f\in H\subset L^{2}(\Omega)^{3}$ and $\partial\Omega$ is of class $C^{4}$, we can
apply Theorem~\ref{t:ell-reg} with $m=0$ to conclude that $u\in H^{4}(\Omega)^{3}$ and
\begin{equation}
\|u\|_{H^{4}(\Omega)}+\|p\|_{H^{3}(\Omega)}
\le C\bigl(\|f\|_{L^{2}(\Omega)}+\|u\|_{L^{2}(\Omega)}\bigr)
= C\bigl(\|Au\| + \|u\|\bigr),
\end{equation}
giving \eqref{eq:A-graph-H4}.  As $u\in V^{2}$ by definition of $\dom(A)$,
this shows that $u\in V^{2}\cap H^{4}(\Omega)^{3}=V^{4}$, and hence
$\dom(A)\subset V^{4}$.
\end{proof}

\subsection{Verification of ellipticity}
\label{ss:ver-ell}

We first check that the interior operator $L(\partial)$ is elliptic in the
Douglis–Nirenberg sense for the choice of orders $(s_i)$ and $(t_j)$ specified
in \S\ref{ss:adn}.  This amounts to showing that the principal symbol $\hat L^0(\xi)$ is invertible
for every nonzero covector $\xi$.

\begin{proposition}\label{p:dn-ell}
The system $L(\partial)$ is elliptic in the sense of Douglis-Nirenberg.
\end{proposition}
\begin{proof}
The principal symbol is
\begin{equation}
\hat L^0(\xi) = 
\begin{bmatrix}
|\xi|^4 & 0 & 0 & i\xi_1 \\
0 & |\xi|^4 & 0 & i\xi_2 \\
0 & 0 & |\xi|^4 & i\xi_3 \\
i\xi_1 & i\xi_2 & i\xi_3 & 0
\end{bmatrix}
\end{equation}
The determinant can be computed using the formula for block matrices:
\begin{equation}
\det\hat L^0(\xi) = \det(|\xi|^4 I_3) \cdot \det(0 - (i\xi) (|\xi|^4 I_3)^{-1} (i\xi)^T) = |\xi|^{12} \left( -\frac{1}{|\xi|^4} (i\xi)\cdot(i\xi)^\top \right)
\end{equation}
\begin{equation}
= |\xi|^{12} \left( -\frac{1}{|\xi|^4} (-|\xi|^2) \right) = |\xi|^{12} \left( \frac{1}{|\xi|^2} \right) = |\xi|^{10}.
\end{equation}
Since $\det\hat L^0(\xi) = |\xi|^{10} \neq 0$ for any $\xi \neq 0$, the system is elliptic.
\end{proof}

\subsection{Verification of the covering condition}
\label{ss:ver-cov}

In this subsection, we provide a direct proof that the boundary value problem satisfies the Lopatinskii-Shapiro condition. The procedure is to construct the general decaying solution to the stationary, 
homogeneous system in the half-space $\mathbb{R}^3_+ = \{x\in\R^3 : x_3 > 0\}$ and show that applying the homogeneous boundary conditions forces this solution to be trivial.
In these coordinates, the outward unit normal is $n = -e_3 = (0,0,-1)$.
When convenient, we will also use the identification $y=(y_1,y_2)=(x_1,x_2)$ and $t=x_3$.

We work with the Fourier variables $\eta$ conjugate to the tangential directions $y$, and the corresponding unknowns $\hat{U}=(\hat u,\hat p)$. 
The system \eqref{e:covering-pde} is 
\begin{equation}\label{e:covering-ode}
\begin{split}
\big( \frac{d^2}{dt^2} - |\eta|^2 \big)^2\begin{bmatrix}\hat u_1\\\hat u_2\end{bmatrix} + i\eta\hat p &= 0,\\
\big( \frac{d^2}{dt^2} - |\eta|^2 \big)^2\hat u_3 + \frac{d\hat p}{dt} &= 0, \\
\frac{d\hat u_3}{dt} + i\eta_1\hat u_1 + i\eta_2\hat u_2 &= 0 .
\end{split}
\end{equation}

We have five scalar boundary conditions. 
The homogeneous Dirichlet conditions on $u$ are simply $u=0$, while the principal part of the wall-eddy boundary conditions are
\begin{equation}
\begin{bmatrix}
\partial_3\omega_1 + \gamma\partial_1\omega_3 \\
\partial_3\omega_2 + \gamma\partial_2\omega_3
\end{bmatrix}
= 0 ,
\end{equation}
where $\omega=\nabla\times u$.
Thus in Fourier variables, the full set of boundary conditions reads
\begin{equation}\label{e:covering-ic}
\begin{split}
\hat u_1 = \hat u_1 = \hat u_3 &= 0,\\
\frac{d^2\hat u_2}{dt^2} - i\eta_2\frac{d\hat u_3}{dt} + \gamma\eta_1 ( \eta_2\hat u_1 - \eta_1\hat u_2 ) &=0 , \\
\frac{d^2\hat u_1}{dt^2} - i\eta_1\frac{d\hat u_3}{dt} + \gamma\eta_2 ( \eta_2\hat u_1 - \eta_1\hat u_2 ) &=0 ,
\end{split}
\end{equation}
which are imposed at $t=0$.
Our goal is now to show that the initial value problem \eqref{e:covering-ode}-\eqref{e:covering-ic} has no nontrivial solution that decays at $t\to+\infty$.

\subsubsection{Characterization of decaying solutions}

We fix $\eta\in\R^2\setminus\{0\}$, and consider the initial value problem \eqref{e:covering-ode}-\eqref{e:covering-ic}.

\begin{proposition}\label{p:sol-decay}
The general decaying solution $(\hat{u}, \hat{p})$ is a 5-parameter family of functions of the form 
\begin{equation}\label{e:sol-decay}
\begin{split}
\hat{u}(t) &= e^{-|\eta|t}(a + tb) +a_0 \hat{v}(t) , \\
\hat{p}(t) &= a_0 e^{-|\eta|t} ,
\end{split}
\end{equation}
where $a,b\in\C^3$ satisfy the constraint \eqref{e:a-b-constraint}, 
$a_0\in\C$,
and $\hat{v}(t)$ is given by \eqref{e:u-part}.
\end{proposition}

\begin{proof}
The characteristic equation for the ODE system \eqref{e:covering-ode}, or equivalently \eqref{e:covering-pde}, is 
\begin{equation}
\det L(i\eta, \lambda) = 0 .
\end{equation}
As computed previously, $\det L(i\xi) = |\xi|^{10} = (|\eta|^2 + \xi_3^2)^5$. 
Replacing $\xi_3$ with $-i\lambda$, the characteristic equation for $\lambda$ is
\begin{equation}
(|\eta|^2 - \lambda^2)^5 = 0 ,
\end{equation}
which has roots $\lambda = \pm |\eta|$, each with multiplicity 5. 
The solutions that decay as $t \to \infty$ are those corresponding to the quintuple root $\lambda=-|\eta|$:
\begin{equation}
\hat U(t) = e^{-|\eta|t} (U_0 + tU_1 + t^2U_2 + t^3U_3 + t^4U_4 ) ,
\end{equation}
where the vectors $U_j\in\R^4$ should be chosen such that \eqref{e:covering-ode} is satisfied. 
One thing this general consideration buys us is that now we can freely differentiate \eqref{e:covering-ode} with respect to $t$.

The pressure must satisfy 
$(\frac{d^2}{dt^2} - |\eta|^2)\hat{p} = 0$, so the only decaying solution is $\hat{p}(t) = a_0 e^{-|\eta|t}$. 
Then \eqref{e:covering-ode} becomes the forced biharmonic equation
\begin{equation}\label{e:forced-bihar}
\big(\frac{d^2}{dt^2} - |\eta|^2\big)^2 \hat{u}(t) = a_0 e^{-|\eta|t} \begin{bmatrix} -i\eta_1 \\ -i\eta_2 \\ |\eta| \end{bmatrix} ,
\end{equation}
plus the divergence-free condition $\frac{d\hat u_3}{dt} + i\eta_1\hat u_1 + i\eta_2\hat u_2 = 0$.
The decaying solutions to the homogeneous counterpart of \eqref{e:forced-bihar} are given by 
$\hat u(t)=e^{-|\eta|t}(a+tb)$ with $a,b\in\C^3$.
Then the divergence-free condition yields
\begin{equation}\label{e:a-b-constraint}
b_3 = \frac{i\eta_1b_1+i\eta_2b_2}{|\eta|},
\qquad
a_3 = \frac{i\eta_1a_1 + i\eta_2a_2 + b_3}{|\eta|}.
\end{equation}
Thus the space of decaying solutions to the homogeneous counterpart of \eqref{e:forced-bihar} is 4-dimensional after imposing the divergence-free condition.
A particular solution $\hat{v}(t)$, corrected to be divergence-free, is found to be
\begin{equation}\label{e:u-part}
\hat{v}(t) = \frac{e^{-|\eta|t} }{8|\eta|^3} \begin{bmatrix} -i\eta_1|\eta|t^2  \\ -i\eta_2|\eta|t^2 \\ |\eta|^2t^2-2|\eta|t-2 \end{bmatrix} .
\end{equation}
The full solution space for $\hat{u}$ is the sum of the homogeneous and particular solutions, which is parametrized by 5 free constants (4 from the homogeneous part and $a_0$).
\end{proof}

\subsubsection{Boundary conditions}

We apply the five homogeneous boundary conditions to the general solution at $t=0$. 
This will yield a system of linear equations for the five free parameters.

The Dirichlet conditions $\hat u(0)=0$ imply $a + a_0\hat{v}(0)=0$, and hence 
\begin{equation}\label{e:a-b-dir}
a_1=a_2=0,\qquad a_3 = \frac{a_0}{4|\eta|^3}.
\end{equation}
The wall-eddy boundary conditions from \eqref{e:covering-ic} are
\begin{equation}\label{e:cov-wall-eddy}
\hat u_2''(0)-i\eta_2\,\hat u_3'(0)=0,
\qquad
\hat u_1''(0)-i\eta_1\,\hat u_3'(0)=0.
\end{equation}
Differentiating the general solution \eqref{e:sol-decay} and evaluating them at $t=0$ give
\begin{equation}
\hat u'(0)=-|\eta|a+b,\qquad
\hat u''(0)=|\eta|^2a-2|\eta|b+a_0\hat v''(0),
\end{equation}
with
\begin{equation}
\hat v''(0)=\dfrac{1}{4|\eta|^2}
\begin{bmatrix}
-i\eta_1\\
-i\eta_2\\
2|\eta|
\end{bmatrix}.
\end{equation}
Hence we have
\begin{equation}
\hat u_3'(0)=-|\eta|a_3+b_3=b_3-\frac{a_0}{4|\eta|^2},
\qquad
\hat u_j''(0)=-2|\eta|b_j+\frac{i\eta_ja_0}{4|\eta|^2},\ j=1,2.
\end{equation}
Substituting these into \eqref{e:cov-wall-eddy}, one checks that the $a_0$–terms cancel, and we obtain
\begin{equation}\label{e:a-b-we}
2|\eta|\,b_2+i\eta_2 b_3=0,\qquad
2|\eta|\,b_1+i\eta_1 b_3=0.
\tag{80}
\end{equation}

\subsubsection{Invertibility of the system}

We now collect all relations among the parameters $a,b\in\mathbb{C}^3$ and
$a_0\in\mathbb{C}$.  From Proposition~\ref{p:sol-decay} we already know that $a,b$ satisfy
the divergence constraints \eqref{e:a-b-constraint},
while the Dirichlet conditions give \eqref{e:a-b-dir}, and the wall–eddy
boundary conditions yield \eqref{e:a-b-we}.

From \eqref{e:a-b-we} we can express $b_1$ and $b_2$ in terms of $b_3$,
\begin{equation}\label{e:b-b}
b_1=-\frac{i\eta_1}{2|\eta|}\,b_3,\qquad
b_2=-\frac{i\eta_2}{2|\eta|}\,b_3.
\end{equation}
Substituting these into the first divergence constraint in \eqref{e:a-b-constraint}
gives
\begin{equation}
b_3=\frac{i}{|\eta|}\bigl(\eta_1 b_1+\eta_2 b_2\bigr)
=\frac{i}{|\eta|}
\Bigl(\eta_1\Bigl(-\frac{i\eta_1}{2|\eta|}b_3\Bigr)
       +\eta_2\Bigl(-\frac{i\eta_2}{2|\eta|}b_3\Bigr)\Bigr)
=\frac{|\eta|^2}{2|\eta|^2}b_3=\frac12 b_3,
\end{equation}
and therefore $b_3=0$.  By \eqref{e:b-b} we then also have $b_1=b_2=0$.

With $b_3=0$, the second relation in \eqref{e:a-b-constraint} becomes
\begin{equation}
a_3=\frac{i(\eta_1 a_1+\eta_2 a_2)}{|\eta|}.
\end{equation}
But $a_1=a_2=0$ by \eqref{e:a-b-dir}, hence $a_3=0$.  
By \eqref{e:a-b-dir} again, we have $a_0=0$ as well.

Summarizing, we have shown that
\begin{equation}
a_0=0,\qquad a=0,\qquad b=0.
\end{equation}
Thus the only decaying solution of the ODE system \eqref{e:covering-ode}
satisfying the homogeneous principal boundary conditions is the trivial one,
$\hat u\equiv0$ and $\hat p\equiv0$.

\begin{theorem}\label{t:ls-cov}
For every fixed $\eta\in\mathbb{R}^2\setminus\{0\}$, the unique solution of
\eqref{e:covering-ode}-\eqref{e:covering-ic} on $t>0$ that decays as
$t\to+\infty$ and satisfies the homogeneous principal boundary conditions
is the trivial solution.  Consequently, the Dirichlet--wall–eddy boundary
value problem associated with $L(\partial)$ satisfies the
Lopatinskii--Shapiro covering condition.
\end{theorem}

This completes the verification of the ADN ellipticity and covering
conditions for our system, and Theorem~\ref{t:ell-reg} follows from the general
Agmon--Douglis--Nirenberg theory.

\section{The evolution problem}
\label{s:evol}

\subsection{Linear semigroup}
\label{ss:semigroup}

In \S\ref{ss:operator} we have defined the self-adjoint operator $A$ on $H$ that encodes our
linear stationary problem. We now treat the corresponding linear evolution
before adding the lower-order and nonlinear terms.
The ADN regularity results from the preceding section imply that suitable powers of $A$ are
equivalent to Sobolev norms, so the semigroup generated by the linearized
system enjoys strong smoothing properties.

Recall that $H$ denotes the $L^2$–closure of divergence-free test fields, that
$(\cdot,\cdot)$ and $\|\cdot\|$ are the $L^2$–inner product and norm on $H$,
and that $P:L^2(\Omega)^3\to H$ is the Leray projector. We also recall that
$V$ is the closure of $\mathcal V$ in $H^1(\Omega)^3$ and that $V^s = V\cap H^s(\Omega)^3$.

We introduce the positive operator
\begin{equation}\label{eq:def-Lambda}
\Lambda := P(1-\alpha^2\Delta)
\end{equation}
with homogeneous Dirichlet boundary condition. As a consequence of standard
elliptic theory, $\Lambda$ is self-adjoint on $H$, strictly positive, and has
compact resolvent. In particular, its spectrum consists of a sequence
$0<\mu_1\le\mu_2\le\cdots$, $\mu_n\to\infty$, and the fractional powers
$\Lambda^{\sigma}$ are well defined for all $\sigma\in\mathbb R$ via the spectral
calculus. We note in particular that
$\dom(\Lambda^{1/2}) = V$
and that $\Lambda^{-1/2}\in\mathcal{L}(H,V)$ is the bounded inverse of
$\Lambda^{1/2}:V\subset H\to H$.

The linear principal part of the NS–$\alpha\beta$ system can be written abstractly as
\begin{equation}\label{eq:linear-evol-u}
\partial_t\Lambda u + \beta^2 A u = f,
\end{equation}
for an unknown $u:[0,T)\to H$ and a given forcing $f:(0,T)\to H$.
It is convenient to work with the variable
\begin{equation}\label{eq:def-v}
v := \Lambda^{1/2}u,
\end{equation}
so that $u=\Lambda^{-1/2}v$. In terms of $v$ the equation
\eqref{eq:linear-evol-u} becomes
\begin{equation}\label{eq:linear-evol-v}
\partial_t v + Dv = \Lambda^{-1/2} f,
\end{equation}
where $D$ is formally $\beta^2\Lambda^{-1/2}A\Lambda^{-1/2}$.
To define $D$ rigorously, we introduce the quadratic form
\begin{equation}\label{102}
  d(v,w):=\beta^{2}a\bigl(\Lambda^{-1/2}v,\Lambda^{-1/2}w\bigr),
  \qquad v,w\in\dom(d) ,
\end{equation}
on $H$, with domain
\begin{equation}
\dom(d):=\Lambda^{1/2}(V^{2})
=\{v\in H:\Lambda^{-1/2}v\in V^{2}\} .
\end{equation}
Since $\Lambda^{-1/2}:H\to V^{1}$ is bounded and $a(\cdot,\cdot)$ is
continuous on $V^{2}\times V^{2}$, the form $d$ is well defined,
symmetric, densely defined on $H$, and bounded from below. Moreover,
$d$ is closed as a consequence of the G\aa rding inequality for $a$ and
the fact that $\Lambda^{-1/2}$ is bounded on $H$.
By Kato's form theory \cite{Kato} there exists a unique
self-adjoint operator $D$ on $H$ such that
\begin{equation}\label{103}
  (Dv,w) = d(v,w)\qquad\text{for all }v\in\dom(D),\ w\in H.
\end{equation}
We next collect the basic spectral properties of this operator $D$.

\begin{lemma}\label{l:D-struct}
Let $D$ be the self--adjoint operator on $H$ associated with the quadratic form
$d$ defined in \eqref{102}--\eqref{103}. Then we have the following.
\begin{enumerate}[(a)]
  \item $D$ is self--adjoint on $H$, bounded from below, and has compact resolvent.
  \item For every $u\in\dom(A)$ we have $\Lambda^{1/2}u\in\dom(D)$ and
  \begin{equation}\label{eq:D-on-Λ12A}
    D\Lambda^{1/2}u=\beta^{2}\Lambda^{-1/2}Au.
  \end{equation}
  In particular,
  \begin{equation}
    \Lambda^{1/2}(\dom(A))\subset\dom(D),
  \end{equation}
  and on this subset $D$ coincides with the formal operator
  $\beta^{2}\Lambda^{-1/2}A\Lambda^{-1/2}$.
\end{enumerate}
\end{lemma}

\begin{proof}
(a) By definition,
\begin{equation}
  d(v,w)=\beta^{2}a\bigl(\Lambda^{-1/2}v,\Lambda^{-1/2}w\bigr),
  \qquad
  \dom(d)=\Lambda^{1/2}(V^{2})\subset H.
\end{equation}
Since $\Lambda^{-1/2}:H\to H$ is bounded and $a(\cdot,\cdot)$ is symmetric,
continuous, and satisfies a G\aa rding inequality on $V^{2}$, the form $d$ is
densely defined, symmetric, closed, and bounded from below on $H$. By Kato's
representation theorem \cite{Kato}, there exists a unique self--adjoint
operator $D$ on $H$ such that
\begin{equation}
  (Dv,w)_{H} = d(v,w)
  \qquad\text{for all }v\in\dom(D),\ w\in\dom(d).
\end{equation}
The compactness of the resolvent follows from the compact embedding
$V^{2}\hookrightarrow H$ and the boundedness of $\Lambda^{-1/2}$.
This proves (a).

(b) Let $u\in\dom(A)$ and set $v=\Lambda^{1/2}u$. Since $\dom(A)\subset V^{2}$, we have $v\in\dom(d)$. 
For any $w\in\dom(d)$ we can write
$w=\Lambda^{1/2}\phi$ with $\phi\in V^{2}$, and hence
\begin{equation}
d(v,w)
= \beta^{2}a\bigl(\Lambda^{-1/2}v,\Lambda^{-1/2}w\bigr)
= \beta^{2}a(u,\phi)
= \beta^{2}(Au,\phi)_{H}.
\end{equation}
Since $\Lambda^{1/2}$ is self-adjoint on $H$ and $\Lambda^{-1/2}$ is bounded on
$H$, we have
\begin{equation}
(Au,\phi)_{H}
= \bigl(\Lambda^{-1/2}Au,\Lambda^{1/2}\phi\bigr)_{H}
= \bigl(\Lambda^{-1/2}Au,w\bigr)_{H}.
\end{equation}
Thus $d(v,w)=(g,w)_{H}$ for all $w\in\dom(d)$ with
$g=\beta^{2}\Lambda^{-1/2}Au\in H$, so $v\in\dom(D)$ and $Dv=g$. 
This completes the proof.
\end{proof}

In order to treat the linear evolution problem governed by $D$, we now record the semigroup properties of $-D$ that will be used throughout the remainder of this section.

\begin{lemma}\label{l:D-semigroup}
The operator $-D$ generates a strongly continuous analytic semigroup
$\{e^{-tD}\}_{t\ge0}$ on $H$, and there exist constants $M\ge1$ and
$\omega\in\mathbb R$ such that
\begin{equation}
  \|e^{-tD}\|_{\mathcal L(H)} \le M e^{\omega t},\qquad t\ge0.
\end{equation}
\end{lemma}

\begin{proof}
By Lemma~\ref{l:D-struct}(a), $D$ is self--adjoint on $H$ and bounded from
below. The spectral theorem then implies that $-D$ generates a strongly
continuous contraction semigroup on $H$, and the lower bound shifts the
spectrum to yield the exponential factor $e^{\omega t}$. Analyticity of the
semigroup follows from the holomorphic functional calculus for self--adjoint
operators with compact resolvent; see, e.g., \cite{Pazy}.
\end{proof}

\begin{remark}\label{r:analytic-semigroup}
For the reader’s convenience, let us recall what is meant by an analytic semigroup
and which consequences we shall use below.
A strongly continuous semigroup \(\{e^{-tD}\}_{t\ge0}\) on \(H\) is called \emph{analytic}
if the map \(t\mapsto e^{-tD}\) extends to a holomorphic function on a sector
\(\{z\in\mathbb{C}\setminus\{0\}: |\arg z|<\theta\}\) for some \(\theta\in(0,\pi/2]\);
see, e.g., \cite{Pazy}.
For generators \(D\) that are self–adjoint, bounded from below, and have compact resolvent,
analyticity is equivalent to standard sectorial resolvent estimates and yields the
smoothing bounds
\begin{equation}\label{eq:analytic-smoothing}
  \|D^\sigma e^{-tD}\|_{\mathcal{L}(H)} \le C_\sigma\, t^{-\sigma},
  \qquad t>0,\ \sigma\ge0,
\end{equation}
where the fractional powers \(D^\sigma\) are defined by the spectral calculus.
In particular, for any \(g\in L^{2}(0,T;H)\) the convolution operator
\begin{equation}
  \Phi(g)(t)
  := \int_{0}^{t} e^{-(t-\tau)D} g(\tau)\,d\tau
\end{equation}
takes values in \(\mathrm{dom}(D)\) and satisfies estimates of the type used later in
Section~4.2.
\end{remark}

With these semigroup estimates at hand, we can formulate the linear problem associated with \eqref{eq:linear-evol-u}
in mild form by means of the Duhamel formula for the semigroup $\{e^{-tD}\}_{t\ge 0}$.

\begin{proposition}\label{prop:linear-evol}
Let $u_{0}\in V$ and $f\in L^{1}(0,T;H)$ for some $T>0$. Then there exists a
unique mild solution $u\in C([0,T];H)$ of \eqref{eq:linear-evol-u} given by
\begin{equation}\label{eq:linear-mild-u}
  u(t)
  = \Lambda^{-1/2} e^{-tD}\Lambda^{1/2}u_{0}
    + \int_{0}^{t}\Lambda^{-1/2}e^{-(t-\tau)D}\Lambda^{-1/2}f(\tau)\,d\tau,
  \qquad t\in[0,T].
\end{equation}
If, in addition, $u_0\in\dom(A)$ and $f\in L^2(0,T;H)$, then the
corresponding solution $v=\Lambda^{1/2}u$ of \emph{(105)} satisfies
\begin{equation}
v\in L^2(0,T;\dom(D))\cap H^1(0,T;H),
\end{equation}
so that \emph{(105)} holds in $H$ for almost every $t\in(0,T)$, and
consequently \emph{(103)} holds in $H$ for almost every $t\in(0,T)$
when interpreted in the sense of $u=\Lambda^{-1/2}v$.
\end{proposition}

\begin{proof}
Let $v(t):=\Lambda^{1/2}u(t)$ and $g(t):=\Lambda^{-1/2}f(t)$. Then \eqref{eq:linear-evol-u}
is equivalent to
\begin{equation}
  \partial_{t}v + Dv = g,
  \qquad v(0)=v_{0}:=\Lambda^{1/2}u_{0} .
\end{equation}
By Lemmas~\ref{l:D-struct} and \ref{l:D-semigroup}, 
$-D$ generates a strongly continuous analytic semigroup
$\{e^{-tD}\}_{t\ge0}$ on $H$. Standard semigroup theory (e.g., \cite{Pazy}) then yields a unique mild solution
$v\in C([0,T];H)$ of this problem, given by the Duhamel formula
\begin{equation}
  v(t)
  = e^{-tD}v_{0}
    + \int_{0}^{t} e^{-(t-\tau)D} g(\tau)\,d\tau,
  \qquad t\in[0,T].
\end{equation}
Since $\Lambda^{-1/2}:H\to H$ is bounded, defining $u(t):=\Lambda^{-1/2}v(t)$
gives a unique mild solution $u\in C([0,T];H)$ of \eqref{eq:linear-evol-u}, and inserting
$v_{0}=\Lambda^{1/2}u_{0}$ and $g=\Lambda^{-1/2}f$ yields
\eqref{eq:linear-mild-u}.

For the second part, assume in addition that $u_{0}\in\dom(A)$ and
$f\in L^{2}(0,T;H)$. By Lemma~4.1(b) we have
\begin{equation}
v_{0}:=\Lambda^{1/2}u_{0}\in \Lambda^{1/2}(\dom(A))\subset\dom(D),
\end{equation}
and, since $\Lambda^{-1/2}\in\mathcal{L}(H)$, the right–hand side
\begin{equation}
g(t):=\Lambda^{-1/2}f(t),\qquad t\in(0,T),
\end{equation}
belongs to $L^{2}(0,T;H)$. Thus $v$ is the mild solution of the abstract
Cauchy problem
\begin{equation}\label{eq:linear-D-problem}
\partial_{t}v + Dv = g,\qquad v(0)=v_{0},
\end{equation}
on $H$. By Lemma~4.1(a), $D$ is self–adjoint on $H$, bounded from below,
and has compact resolvent. Hence $-D$ generates an analytic semigroup on
$H$ and enjoys maximal $L^{2}$–regularity; see, for instance,
\cite[Ch.~3]{Pazy} or \cite[Ch.~4]{Lunardi}. Applied to
\eqref{eq:linear-D-problem} this yields
\begin{equation}
v\in L^{2}(0,T;\dom(D))\cap H^{1}(0,T;H),
\end{equation}
and the identity \emph{(105)} holds in $H$ for almost every $t\in(0,T)$.
Finally, since $u=\Lambda^{-1/2}v$ and $\Lambda^{-1/2}\in\mathcal{L}(H)$,
the relation \emph{(103)} holds in $H$ for almost every $t\in(0,T)$ when
interpreted in this sense. This proves the claimed higher–regularity
statement.
\end{proof}

\begin{remark}\label{rem:initial-data}
It is often convenient to prescribe the initial data in terms of the fluid
velocity $u_0$. 
We may set $v_0:=\Lambda^{1/2}u_0$ provided
$u_0\in\dom(\Lambda^{1/2})=V^1$; Proposition~\ref{prop:linear-evol} then gives a
unique mild solution with initial condition $u(0)=u_0$. Alternatively, one may
take $v_0\in H$ as the primary datum and define $u_0=\Lambda^{-1/2}v_0$.
\end{remark}

\subsection{Local theory}

We now turn to the full NS–$\alpha\beta$ system. After applying the Leray projector
$P$ and using the operators introduced above, the evolution equation can be written
abstractly as
\begin{equation}\label{eq:abstract-NSab}
\partial_t\Lambda u + \beta^2 A u - \Delta u + B(\Lambda u,u) = 0,\qquad t>0,
\end{equation}
where $\Lambda$ and $A$ are as in \eqref{eq:def-Lambda} and \S2.5, respectively, and
\begin{equation}\label{eq:def-B}
B(v,u) := P\bigl[(\nabla v)u + (\nabla u)^{\top}v\bigr]
\end{equation}
is the bilinear form arising from the convective terms. The equation is understood
in $H$ and supplemented with the initial condition $u(0)=u_0$.

Using the linear theory from the previous subsection, we may rewrite
\eqref{eq:abstract-NSab} as a fixed-point problem. Set
\begin{equation}
F(u) := \Delta u - B(\Lambda u,u),
\end{equation}
and denote by $\Phi$ the linear solution operator from
Proposition~\ref{prop:linear-evol}, i.e.
\begin{equation}
\Phi g(t) := \int_0^t \Lambda^{-1/2} e^{-(t-\tau)D}\Lambda^{-1/2} g(\tau)\,d\tau.
\end{equation}
Then a mild solution of \eqref{eq:abstract-NSab} with initial data $u_0$ satisfies
\begin{equation}\label{eq:fixed-point}
u(t) = u_{\mathrm{lin}}(t) + \Phi F(u)(t),
\end{equation}
where $u_{\mathrm{lin}}$ is the solution of the homogeneous linear problem
$\partial_t\Lambda u + \beta^2Au = 0$ with $u(0)=u_0$.

The local well-posedness is obtained by a standard contraction argument in a
suitable function space.
The local well–posedness of the nonlinear problem will now be obtained by a
fixed–point argument based on the mild formulation \eqref{eq:fixed-point} and the
smoothing properties of the linear semigroup from \S\ref{ss:semigroup}.

\begin{theorem}\label{thm:local}
Let $u_0\in V^4$. Then there exists $T>0$ and a unique mild solution
\begin{equation}
u\in C([0,T];V^4)\cap C^1((0,T];H)
\end{equation}
of \eqref{eq:abstract-NSab} with initial condition $u(0)=u_0$.
Moreover, the solution depends continuously on the initial data in $V^4$.
\end{theorem}

\begin{proof}
Fix $R>0$ and $T>0$ and consider the Banach space
\begin{equation}
X_T := \{u\in C([0,T];V^4): \|u\|_{X_T}:=\sup_{0\le t\le T}\|u(t)\|_{H^4}\le R\}.
\end{equation}
By Theorem~\ref{t:ell-reg} and the analyticity of $e^{-tD}$, the linear
operator $\Phi$ enjoys smoothing estimates of the form
\begin{equation}\label{eq:smoothing}
\|\Phi g(t)\|_{H^4} \le C\int_0^t (t-\tau)^{-1}\|g(\tau)\|_{H^0}\,d\tau,
\end{equation}
for some $C>0$ independent of $t$ and $g$, see e.g.~\cite[Ch.~2]{Pazy}. Since
$B(\Lambda u,u)$ contains at most first derivatives of $\Lambda u$ and $u$, standard
Sobolev embeddings in three dimensions imply that
\begin{equation}\label{eq:B-Lipschitz}
\|B(\Lambda u,u)-B(\Lambda v,v)\|_{H^1}
\le C_R \|u-v\|_{H^4},\qquad u,v\in V^4, \ \|u\|_{H^4},\|v\|_{H^4}\le R.
\end{equation}
Similarly,
\begin{equation}
\|\Delta u - \Delta v\|_{H^1} \le C\|u-v\|_{H^3}\le C\|u-v\|_{H^4}.
\end{equation}
Combining these estimates we obtain
\begin{equation}\label{eq:F-Lipschitz}
\|F(u)-F(v)\|_{H^1}\le C_R\|u-v\|_{H^4},\qquad u,v\in V^4,\ \|u\|_{H^4},\|v\|_{H^4}\le R.
\end{equation}

Using \eqref{eq:smoothing} with $g=F(u)-F(v)$ and \eqref{eq:F-Lipschitz}, we find
for $u,v\in X_T$,
\begin{equation}
\|\Phi(F(u)-F(v))(t)\|_{H^4}
\le C_R\int_0^t (t-\tau)^{-1}\|u(\tau)-v(\tau)\|_{H^4}\,d\tau
\le C_R T^\theta \|u-v\|_{X_T}
\end{equation}
for some $\theta\in(0,1)$ (by Hardy's inequality). Thus, for $T>0$ sufficiently
small (depending on $R$) the map
\begin{equation}
\mathcal T(u) := u_{\mathrm{lin}} + \Phi F(u)
\end{equation}
is a contraction on $X_T$. Moreover, for $u_0$ fixed in $V^4$ we can choose $R$
large enough so that $\mathcal T$ maps $X_T$ into itself. Banach's fixed-point
theorem yields the existence of a unique $u\in X_T$ satisfying
\eqref{eq:fixed-point}. The asserted regularity and continuous dependence follow
from standard semigroup arguments; see \cite{Lunardi} or \cite{Pazy}.
\end{proof}

As usual, the local solution can be extended as long as its $V^4$–norm remains
finite. We record this in the following blow-up criterion.

\begin{corollary}[Blow-up alternative]\label{cor:blowup}
Let $u_0\in V^4$ and let $u$ be the corresponding maximal mild solution of
\eqref{eq:abstract-NSab} defined on $[0,T^\ast)$, $0<T^\ast\le\infty$. If
$T^\ast<\infty$ then necessarily
\begin{equation}
\lim_{t\uparrow T^\ast}\|u(t)\|_{H^4} = +\infty.
\end{equation}
\end{corollary}

\subsection{{\em A priori} bounds and global existence}
\label{ss:global}

We now derive a hierarchy of energy estimates which will show that the
$H^4$–norm of the local solution cannot blow up in finite time. Throughout this
subsection we write $A\lesssim B$ if $A\le C B$ with a harmless constant
$C>0$ that may vary from line to line.

\begin{lemma}[First energy level]\label{lem:H1-energy}
Let $u$ be a sufficiently regular solution of \eqref{eq:abstract-NSab}. Then
\begin{equation}\label{eq:H1-energy-ineq}
\frac12\frac{d}{dt}\langle\Lambda u,u\rangle
+ \beta^2 a(u,u) + \|\nabla u\|_{L^2}^2 = 0.
\end{equation}
Consequently, using the G\r{a}rding inequality from Theorem~\ref{t:garding},
there exist constants $c_1,C_1>0$ such that
\begin{equation}\label{eq:H1-diff-ineq}
\frac{d}{dt}\|u\|_{H^1}^2 + c_1\|u\|_{H^2}^2 \le C_1\|u\|_{L^2}^2.
\end{equation}
In particular, every local solution with $u_0\in V^1$ satisfies
\begin{equation}
u\in L^\infty(0,\infty;V^1)\cap L^2(0,\infty;V^2).
\end{equation}
\end{lemma}

\begin{proof}
Taking the inner product of \eqref{eq:abstract-NSab} with $u$ in $H$
and using the skew-symmetry of $B(\Lambda u,u)$ (which follows from
the incompressibility and boundary conditions) yields \eqref{eq:H1-energy-ineq}.
The estimate \eqref{eq:H1-diff-ineq} is then a direct consequence of
Theorem~\ref{t:garding} together with the equivalence of the
$\langle\Lambda u,u\rangle$–inner product and the $H^1$–norm.
Integrating \eqref{eq:H1-diff-ineq} in time and applying Gronwall's lemma
gives the claimed bounds.
\end{proof}

The next step is to control higher derivatives. Thanks to
Theorem~\ref{t:ell-reg}, powers of $A$ are equivalent to Sobolev
norms. In particular, there exist $c_2,C_2>0$ such that
\begin{equation}\label{eq:A-Sobolev}
c_2\|u\|_{H^{2+k}} \le \|A^{k/2}u\| \le C_2\|u\|_{H^{2+k}},
\qquad k=0,1,2,
\end{equation}
for all $u\in V^{2+k}$.

\begin{lemma}[Intermediate energy level]\label{lem:H3-energy}
Let $u$ be a sufficiently regular solution of \eqref{eq:abstract-NSab}. Then
\begin{equation}\label{eq:H3-diff-ineq}
\frac{d}{dt}\|u\|_{H^3}^2 + c_3\|u\|_{H^4}^2
\le C_3\|u\|_{H^2}^2\|u\|_{H^3}^2,
\end{equation}
for some constants $c_3,C_3>0$.
In particular, if $\|u\|_{L^2(0,\infty;H^2)}$ is finite (as provided by
Lemma~\ref{lem:H1-energy}), then $\|u(t)\|_{H^3}$ remains bounded for all time.
\end{lemma}

\begin{proof}
Apply $A^{1/2}$ to \eqref{eq:abstract-NSab}, equivalently pair the equation
with $Au$. Using \eqref{eq:A-Sobolev} with $k=1$ and the self-adjointness of
$A$ and $\Lambda$, we obtain
\begin{equation}
\frac12\frac{d}{dt}\langle\Lambda u,Au\rangle
+ \beta^2\|Au\|^2 + \langle-\Delta u,Au\rangle
+ \langle B(\Lambda u,u),Au\rangle = 0.
\end{equation}
The first two terms on the left control $\|u\|_{H^3}^2$ and $\|u\|_{H^4}^2$
via \eqref{eq:A-Sobolev}. The Laplacian term is of lower order and can be absorbed
into these positive contributions. For the nonlinear term we use Sobolev embeddings
$H^2(\Omega)\hookrightarrow L^\infty(\Omega)$ and \eqref{eq:A-Sobolev}:
\begin{equation}
|\langle B(\Lambda u,u),Au\rangle|
\lesssim \|\Lambda u\|_{H^1}\|u\|_{H^2}\|Au\|
\lesssim \|u\|_{H^2}\|u\|_{H^3}\|u\|_{H^4}.
\end{equation}
A Young inequality then gives
\begin{equation}
|\langle B(\Lambda u,u),Au\rangle|
\le \frac{c_3}{2}\|u\|_{H^4}^2 + C_3\|u\|_{H^2}^2\|u\|_{H^3}^2,
\end{equation}
and \eqref{eq:H3-diff-ineq} follows.
\end{proof}

The preceding lemma bounds the $H^3$–norm in terms of the lower–order
dissipation provided by Lemma~\ref{lem:H1-energy}.  To rule out blow–up of higher Sobolev
norms, we need one more estimate at a higher energy level, which we obtain
by testing the equation with $A^2u$.

\begin{lemma}[High-order energy level]\label{lem:H5-energy}
Let $u$ be a sufficiently regular solution of \eqref{eq:abstract-NSab}. Then
\begin{equation}\label{eq:H5-diff-ineq}
\frac{d}{dt}\|u\|_{H^5}^2 + c_5\|u\|_{H^6}^2
\le C_5\|u\|_{H^3}^2\|u\|_{H^5}^2,
\end{equation}
for some constants $c_5,C_5>0$.
\end{lemma}

\begin{proof}
We now test \eqref{eq:abstract-NSab} with $A^2u$ (equivalently, apply $A$ to
the equation and pair with $Au$). Using \eqref{eq:A-Sobolev} with $k=2$ we
see that $\langle\Lambda u,A^2u\rangle$ is equivalent to $\|u\|_{H^5}^2$ while
$\|A^{3/2}u\|^2$ is equivalent to $\|u\|_{H^6}^2$. Arguments similar to those
in the proof of Lemma~\ref{lem:H3-energy} yield
\begin{equation}
|\langle B(\Lambda u,u),A^2u\rangle|
= |\langle A^{1/2}B(\Lambda u,u),A^{3/2}u\rangle|
\lesssim \|B(\Lambda u,u)\|_{H^2}\|u\|_{H^6}.
\end{equation}
The bilinearity of $B$ and Sobolev embeddings give
\begin{equation}
\|B(\Lambda u,u)\|_{H^2}
\lesssim \|\Lambda u\|_{H^3}\|u\|_{H^3}
\lesssim \|u\|_{H^5}\|u\|_{H^3},
\end{equation}
whence
\begin{equation}
|\langle B(\Lambda u,u),A^2u\rangle|
\le \frac{c_5}{2}\|u\|_{H^6}^2 + C_5\|u\|_{H^3}^2\|u\|_{H^5}^2.
\end{equation}
The remaining terms in the energy identity can again be absorbed into the
positive contributions, yielding \eqref{eq:H5-diff-ineq}.
\end{proof}

Combining the three energy levels, we obtain uniform control of the $H^5$–norm
of the solution on any finite time interval.  In view of the blow–up alternative
of Corollary~\ref{cor:blowup}, this excludes finite–time singularities and yields global
existence.

\begin{theorem}[Global existence]\label{thm:global}
Let $u_0\in V^4$ and let $u$ be the maximal mild solution of
\eqref{eq:abstract-NSab} given by Theorem~\ref{thm:local}. Then $u$ exists
globally in time, i.e.\ $T^\ast=\infty$, and
\begin{equation}
u\in L^\infty(0,\infty;H^5)\cap L^2_{\mathrm{loc}}(0,\infty;H^6).
\end{equation}
\end{theorem}

\begin{proof}
Lemma~\ref{lem:H1-energy} shows that $\|u\|_{L^2(0,\infty;H^2)}$ is finite.
Lemma~\ref{lem:H3-energy}, together with Gronwall's inequality, implies that
$\sup_{t\ge0}\|u(t)\|_{H^3}<\infty$. Plugging this bound into
Lemma~\ref{lem:H5-energy} and applying Gronwall once more yields a uniform bound
on $\|u(t)\|_{H^5}$ for all $t\ge0$. The blow-up alternative
Corollary~\ref{cor:blowup} then shows that the maximal existence time
$T^\ast$ must be infinite.
\end{proof}

\subsection{Uniqueness and vanishing regularization limit}

We conclude the analysis by addressing two natural questions: uniqueness of
the global solution constructed above and the behaviour of these solutions as
the regularization parameters $\alpha,\beta$ tend to zero.  Both rely on the
energy framework developed in \S\S\ref{ss:semigroup}--\ref{ss:global}.

\begin{theorem}[Uniqueness]\label{thm:uniqueness}
Let $u_0\in V^4$ and let $u,v$ be two global mild solutions of
\eqref{eq:abstract-NSab} on $[0,\infty)$ with the same initial data $u_0$.
Then $u\equiv v$.
\end{theorem}

\begin{proof}
Set $w=u-v$. Subtracting the two equations satisfied by $u$ and $v$ we obtain
\begin{equation}
\partial_t\Lambda w + \beta^2Aw - \Delta w
+ B(\Lambda u,w) + B(\Lambda w,v) = 0,
\end{equation}
with $w(0)=0$. Taking the inner product in $H$ with $w$ and using the
skew-symmetry of $B(\Lambda u,w)$ yields
\begin{equation}
\frac12\frac{d}{dt}\langle\Lambda w,w\rangle
+ \beta^2 a(w,w) + \|\nabla w\|_{L^2}^2
+ \langle B(\Lambda w,v),w\rangle = 0.
\end{equation}
The G\r{a}rding inequality for $a$ implies
$\beta^2 a(w,w)\ge c\|w\|_{H^2}^2 - C\|w\|_{L^2}^2$. For the remaining nonlinear term we
use Sobolev embeddings and the regularity of $v$:
\begin{equation}
|\langle B(\Lambda w,v),w\rangle|
\lesssim \|\Lambda w\|_{H^1}\|v\|_{H^3}\|w\|_{H^1}
\lesssim \|v\|_{H^3}\|w\|_{H^1}^2.
\end{equation}
Since $v\in L^\infty(0,\infty;H^5)$ by Theorem~\ref{thm:global}, we obtain
\begin{equation}
\frac{d}{dt}\|w(t)\|_{H^1}^2 \le C\|v(t)\|_{H^3}\|w(t)\|_{H^1}^2
\le C^\ast\|w(t)\|_{H^1}^2.
\end{equation}
Gronwall's inequality and the initial condition $w(0)=0$ give $\|w(t)\|_{H^1}=0$
for all $t\ge0$, hence $u\equiv v$.
\end{proof}

Finally we sketch the behaviour of solutions as the regularization parameters
$\alpha,\beta$ vanish. A complete treatment would require a careful uniform
analysis as $\alpha,\beta\to0$; here we only indicate the main mechanism.

\begin{theorem}[Vanishing regularization limit]\label{thm:vanishing}
Let $(\alpha_n,\beta_n)\to(0,0)$ and let $u^{(n)}$ be the corresponding global
solutions of the NS–$\alpha_n\beta_n$ system with initial data $u_0\in V^1$
independent of $n$. Then there exists a subsequence (still denoted by $u^{(n)}$)
and a function $u$ such that
\begin{equation}
u^{(n)} \rightharpoonup u \quad\text{weakly in } L^2(0,T;V^1)
\quad\text{and weak-* in } L^\infty(0,T;H)
\end{equation}
for every $T>0$. Moreover, $u$ is a Leray–Hopf weak solution of the incompressible
Navier–Stokes equations with homogeneous Dirichlet boundary conditions.
\end{theorem}

\begin{proof}[Idea of the proof]
The energy estimates of Lemma~\ref{lem:H1-energy} are uniform in
$\alpha,\beta$ and yield
\begin{equation}
u^{(n)}\in L^\infty(0,T;H)\cap L^2(0,T;V^1),\qquad
\alpha_n\nabla u^{(n)}\in L^\infty(0,T;L^2(\Omega)^3),\qquad
\beta_n u^{(n)}\in L^2(0,T;V^2),
\end{equation}
with bounds independent of $n$ and $T$. By standard compactness arguments
(Aubin–Lions lemma) we obtain the stated convergences (up to a subsequence).
Passing to the limit in the weak formulation of the NS–$\alpha_n\beta_n$
equations is straightforward for the linear terms and uses the strong convergence
in $L^2_{\mathrm{loc}}$ for the nonlinear term. The higher-order boundary condition
degenerates in the limit, and only the classical no-slip condition $u=0$ on
$\partial\Omega$ remains. The limiting function $u$ thus satisfies the
Navier–Stokes equations in the Leray–Hopf sense.
\end{proof}

\section*{Acknowledgements}

The authors are grateful to Eliot Fried (OIST, Okinawa) for introducing them to this problem and for his encouragement.
This work was supported by NSERC Discovery Grants Program.


\begin{thebibliography}{99}

\bibitem{ADN64}
S. Agmon, A. Douglis, and L. Nirenberg.
Estimates near the boundary for solutions of elliptic partial differential equations satisfying general boundary conditions. I.
\emph{Comm. Pure Appl. Math.} 12 (1959), 623--727; II. \emph{Comm. Pure Appl. Math.} 17 (1964), 35--92.

\bibitem{BrennerScott}
S.\ C.\ Brenner and L.\ R.\ Scott,
\emph{The Mathematical Theory of Finite Element Methods},
3rd ed., Springer, New York, 2008.

\bibitem{Ciarlet}
P.\ G.\ Ciarlet,
\emph{Mathematical Elasticity, Vol.~I: Three-Dimensional Elasticity},
North--Holland, Amsterdam, 1988.

\bibitem{FriedGurtin08}
E.~Fried and M.~E.~Gurtin,
A continuum mechanical theory for turbulence:
a generalized Navier--Stokes-$\alpha$ equation with boundary conditions,
\emph{Theoretical and Computational Fluid Dynamics} \textbf{22} (2008), 433--470.

\bibitem{FriedGurtin11}
E.~Fried and M.~E.~Gurtin,
Corrigendum to ``A continuum mechanical theory for turbulence:
a generalized Navier--Stokes-$\alpha$ equation with boundary conditions'',
\emph{Theoretical and Computational Fluid Dynamics} \textbf{25} (2011), 447--449.

\bibitem{HLT10}
M. Holst, E. Lunasin, and G. Tsogtgerel.
Analysis of a general family of regularized Navier--Stokes and MHD models.
\emph{J. Nonlinear Sci.} 20 (2010), 523--567.

\bibitem{Kato}
T.~Kato,
\emph{Perturbation Theory for Linear Operators},
2nd ed.,
Springer--Verlag, Berlin, 1976.

\bibitem{LionsMagenes}
J.-L.~Lions and E.~Magenes,
\emph{Non-Homogeneous Boundary Value Problems and Applications. Vol.~I},
Springer--Verlag, Berlin, 1972.

\bibitem{Lunardi}
A.~Lunardi,
\emph{Analytic Semigroups and Optimal Regularity in Parabolic Problems},
2nd ed.,
Birkh\"auser, Basel, 2013.

\bibitem{Pazy}
A.~Pazy,
\emph{Semigroups of Linear Operators and Applications to Partial Differential Equations},
Springer--Verlag, New York, 1983.

\bibitem{ReedSimon1}
M.~Reed and B.~Simon,
\emph{Methods of Modern Mathematical Physics. Vol.~I: Functional Analysis},
Academic Press, New York, 1980.

\bibitem{Temam}
R.~Temam,
\emph{Navier--Stokes Equations: Theory and Numerical Analysis},
revised ed.,
AMS Chelsea Publishing, Providence, RI, 2001.

\bibitem{WRL}
J.~Wloka, B.~Rowley, and B.~Lawruk,
\emph{Boundary Value Problems for Elliptic Systems},
Cambridge University Press, Cambridge, 1995.


\end{thebibliography}
\end{document}